\documentclass[reqno,a4paper,12pt]{amsart}

\usepackage[T1]{fontenc}
\usepackage[english]{babel}
\usepackage[babel]{microtype}
\usepackage[a4paper, margin=1.1in]{geometry}
\usepackage[dvipsnames]{xcolor}
\usepackage{amsmath,amssymb,amsthm,amsrefs}
\usepackage{mathrsfs}
\usepackage{mathtools}
\usepackage{commath}
\usepackage{thmtools}
\usepackage{thm-restate}
\usepackage{bbm}
\usepackage{dsfont}
\usepackage{lmodern}
\usepackage{fixcmex}
\usepackage{cases}
\usepackage{enumitem}
\setlist[enumerate,1]{label=(\roman*)}
\usepackage{centernot}
\usepackage{booktabs}
\usepackage{multirow}
\usepackage{comment}
\usepackage{float}
\usepackage{scalerel}
\usepackage[pdftex, pdfborderstyle={/S/U/W 0}]{hyperref}
\hypersetup{
    colorlinks=true,
    linkcolor=blue!85!black,
    citecolor=blue!85!black,
    urlcolor=blue!85!black,
}
\usepackage{cleveref}
\usepackage{etoolbox}
\usepackage{subfiles}
\usepackage{pdfpages}

\numberwithin{equation}{section}
\declaretheoremstyle[
  headfont=\normalfont\bfseries,
  bodyfont=\normalfont,
]{noital}

\declaretheorem[style=plain,numberwithin=section,name=Theorem]{theorem}
\declaretheorem[style=plain,sibling=theorem,name=Proposition]{proposition}
\declaretheorem[style=plain,sibling=theorem,name=Lemma]{lemma}
\declaretheorem[style=plain,sibling=theorem,name=Corollary]{corollary}

\newcommand{\indef}[1]{\emph{#1}}
\newcommand{\defined}{\mathrel{\coloneqq}}

\DeclarePairedDelimiter{\p}{\lparen}{\rparen}

\renewcommand{\leq}{\leqslant}

\renewcommand{\geq}{\geqslant}

\let\oldexists\exists
\let\exists\relax
\DeclareMathOperator{\exists}{\:\!\oldexists}
\let\oldforall\forall
\let\forall\relax
\DeclareMathOperator{\forall}{\:\!\oldforall}

\newcommand{\st}{\mathbin{\colon}}
\undef{\set}
\DeclarePairedDelimiter{\set}{\lbrace}{\rbrace}
\undef{\emptyset}
\newcommand{\emptyset}{\varnothing}

\DeclarePairedDelimiter{\card}{\lvert}{\rvert}

\newcommand{\union}{\mathbin{\cup}}
\newcommand{\inter}{\mathbin{\cap}}

\newcommand{\from}{\colon}
\newcommand{\symdiff}{\mathbin{\Delta}}
\DeclareMathOperator{\embeds}{\hookrightarrow}

\DeclarePairedDelimiter{\floor}{\lfloor}{\rfloor}
\DeclarePairedDelimiter{\ceil}{\lceil}{\rceil}

\DeclareMathOperator{\loglog}{log\,log}
\DeclareMathOperator{\logloglog}{log\,log\,log}

\newcommand{\divides}{\mathrel{\mid}}

\undef{\mod}
\newcommand{\mod}[1]{\ (\mathrm{mod}\ #1)}

\DeclareMathDelimiter{\given}
  {\mathbin}{symbols}{"6A}{largesymbols}{"0C}
\DeclareMathOperator{\Prob}{\mathbb{P}}
\DeclarePairedDelimiterXPP{\prob}[1]
  {\Prob}{\lparen}{\rparen}{}
  {\renewcommand{\given}{\nonscript\;\delimsize\vert\nonscript\;\mathopen{}}#1}
\DeclareMathOperator{\Expec}{\mathbb{E}}
\DeclarePairedDelimiterXPP{\expec}[1]
  {\Expec}{\lparen}{\rparen}{}
  {\renewcommand{\given}{\nonscript\;\delimsize\vert\nonscript\;\mathopen{}}#1}
\DeclareMathOperator{\Var}{Var}
\DeclarePairedDelimiterXPP{\var}[1]
  {\Var}{\lparen}{\rparen}{}
  {\renewcommand{\given}{\nonscript\;\delimsize\vert\nonscript\;\mathopen{}}#1}
\DeclareMathOperator{\Cov}{Cov}
\DeclarePairedDelimiterXPP{\cov}[2]
  {\Cov}{\lparen}{\rparen}{}{#1,#2}

\newcommand{\eps}{\varepsilon}

\newcommand{\FF}{\mathbb{F}}
\newcommand{\NN}{\mathbb{N}}
\newcommand{\RR}{\mathbb{R}}

\newcommand{\cA}{\mathcal{A}}
\newcommand{\cD}{\mathcal{D}}
\newcommand{\cF}{\mathcal{F}}
\newcommand{\cQ}{\mathcal{Q}}
\newcommand{\cL}{\mathcal{L}}
\newcommand{\cM}{\mathcal{M}}
\newcommand{\cP}{\mathcal{P}}
\newcommand{\cS}{\mathcal{S}}

\begin{document}

\title{Improved bounds for the dimension of divisibility}

\author{Victor Souza and Leo Versteegen}

\address{Department of Pure Mathematics and Mathematical Statistics (DPMMS), University of Cambridge, Wilberforce Road, Cambridge, CB3 0WA, United Kingdom}
\email{\{vss28,lvv23\}@cam.ac.uk}

\begin{abstract}
The dimension of a partially-ordered set $P$ is the smallest integer $d$ such that one can embed $P$ into a product of $d$ linear orders.
We prove that the dimension of the divisibility order on the interval $\{1, \dotsc, n\}$ is bounded above by $C(\log n)^2 (\log \log n)^{-2} \log \log \log n$ as $n$ goes to infinity.
This improves a recent result by Lewis and the first author, who showed an upper bound of $C(\log n)^2 (\log \log n)^{-1}$ and a lower bound of $c(\log n)^2 (\log \log n)^{-2}$, asymptotically.
To obtain these bounds, we provide a refinement of a bound of Füredi and Kahn and exploit a connection between the dimension of the divisibility order and the maximum size of $r$-cover-free families.
\end{abstract}

\maketitle


\section{Introduction}

The \indef{dimension} $\dim P$, of a partially ordered set (poset) $P$ is the smallest integer $d$ such that $P$ can be embedded into a product of $d$ linear orders.
This was introduced by Dushnik and Miller~\cite{Dushnik1950-ma} in 1941 and led to the development of the rich dimension theory of orders.

Even though this concept is not new, the dimension of the divisibility order was studied only recently.
For a set $S \subseteq \NN$, denote by $\cD_{S}$ the divisibility order restricted to the set $S$, that is, the partial order $(S,\leq_S)$ with ground set $S$ and where $a \leq_S b$ when $a$ divides $b$.
Naturally, we are mainly interested in the dimension of $\cD_{[n]}$, where $[n] \defined \set{1, \dotsc, n}$.
Recently, the following bounds were obtained by Lewis and Souza~\cite{Lewis2021-em}.
Here and elsewhere, we denote by $\log$ the natural logarithm.

\begin{theorem}
\label{thm:lewissouza}
Let $\cD_{[n]}$ be the divisibility order on the $[n]$.
Then, as $n \to \infty$
\begin{align}
\label{eq:dim-original}
    \p[\big]{1/16 - o(1)} \frac{(\log n)^2}{(\loglog n)^2}
    \leq \dim \cD_{[n]}
    \leq \p[\big]{4 + o(1)} \frac{(\log n)^2}{\loglog n}.
\end{align}
\end{theorem}

Our main result is to improve the upper bound in \Cref{thm:lewissouza}.
Our new bound determines $\dim \cD_{[n]}$ asymptotically up to a multiplicative $\logloglog n$ factor.

\begin{theorem}
\label{thm:main}
Let $\cD_{[n]}$ be the divisibility order on $[n]$.
Then, as $n \to \infty$,
\begin{equation*}
  \dim \cD_{[n]}
  \leq \p[\big]{4 / \log 2 + o(1)} \frac{(\log n)^2 \logloglog n}{(\loglog n)^2}.
\end{equation*}
\end{theorem}

Our main innovation is a refinement of a bound of Füredi an Kahn~\cite{Furedi1986-bu} on the dimension of suborders of the hypercube, see \Cref{thm:kahn-furedi-refined} below.
We also profit from a connection between the dimension of certain suborders of the hypercube and the size of families of subsets, no one contained in the union of $r$ others.

This paper is organised as follows.
In \Cref{sec:poset} we recall some standard definitions in order theory.
In \Cref{sec:cube} we review the theory of the dimension of suborders of the hypercube and provide our refinement over Füredi and Kahn's bound.
In \Cref{sec:cover} we discuss $r$-cover-free families and state which bounds we will use.
In \Cref{sec:downsets} we give a combinatorial argument that implies that we can restrict to the squarefree numbers to compute $\dim \cD_{[n]}$.
In \Cref{sec:divisibility} we export these tools to the divisibility order: these ideas alone are enough to improve the constant in the upper bound in~(\ref{eq:dim-original}).
Finally, we prove \Cref{thm:main} in \Cref{sec:main} and present some final remarks in \Cref{sec:remarks}.


\section{Dimension of partial orders}
\label{sec:poset}

A \indef{poset} is an ordered pair $P = (S, \leq_P)$, where $S$ is a set and $\leq_P$ is a partial order on $S$.
We usually identify a poset with its ground set $S$.
Two elements $a, b \in S$ are \indef{incomparable} if neither $a \leq_P b$ nor $b \leq_P a$ holds.
A \indef{linear order} is a poset in which the elements are pairwise comparable.
Given two posets $P = (S, \leq_P)$, $Q = (S', \leq_Q)$, a \indef{poset embedding} of $P$ into $Q$ is a map $\varphi \from S \to S'$ such that $\varphi(a) \leq_Q \varphi(b)$ if and only if $a \leq_P b$.
We write $P\embeds Q$ to represent either an embedding of $P$ into $Q$ or the existence of such an embedding, depending on context.
A poset $Q = (T, \leq_Q)$ is called a \indef{suborder} of $P = (S, \leq_P)$ if $T \subseteq S$ and the inclusion map is a poset embedding.
A \indef{chain} in a poset is a suborder that is a linear order, and an \indef{antichain} is a suborder in which no two elements are comparable.

Given posets $P_i = (S_i, \leq_{P_i})$, $i \in [k]$, the \indef{product poset} $P = P_1 \times \dotsb \times P_k$ is the order on the product set $S = S_1 \times \dotsb \times S_k$ where $a \leq_P b$ if and only if $a_i \leq_{P_i} b_i$ for all $i \in [k]$.

Consider $\RR$ with its standard order.
The \indef{dimension} of a countable poset $P$, denoted $\dim P$, is equal to the minimum $d$ such that $P \embeds \RR^d$ with the product order.
It follows from this definition that dimension is subadditive and monotone;
i.e., for any two posets $P$ and $Q$, $\dim P \times Q \leq \dim P + \dim Q$, and, if $P \embeds Q$, then $\dim P \leq \dim Q$.

An equivalent definition of dimension can be given in terms of linear extensions.
Given a poset $P = (S, \leq_P)$, a \indef{linear extension} of $P$ is a linear order $L = (S,\leq_L)$ that extends $P$, that is, if $a \leq_P b$, then $a \leq_L b$.
For any poset $P$, a \indef{realiser} of $P$ is a set $\cL$ of linear extensions of $P$ with the property that, for every pair $(a,b) \in P^2$ with $a \not\geq b$, there exists an $L \in \cL$ such that $a \leq_L b$.
Then the dimension of $P$ is the minimum cardinality of a realiser of $P$.

Given a poset $P = (S,\leq_P)$ and a subset $U \subseteq S$, we say tha the element $a \in U$ is \indef{minimal} (or \indef{maximal}) in $U$ if there is no $b \in U$ such that $b \leq_P a$ (or $a \leq_P b)$.
An element $a \in U$ is said to be a \indef{minimum} (or a \indef{maximum}) in $U$ if for every $b \in U$, we have $a \leq_P b$ (or $b \leq_P a$).
Note that if $x$ is a minimum element of $P$, then $x$ has to be the minimum element in each linear extension of $P$.
Therefore, removing $x$ from $P$ does not change the dimension, unless $\card{P} = 1$.


\section{Dimension of suborders of the cube}
\label{sec:cube}

Before considering the divisibility order, it will be useful to review the classical problem of determining the dimension of subsets of the hypercube $\cQ_{[n]}$.
Indeed, what follows is a brief summary of the survey of Kierstead~\cite{Kierstead1999-vj}. 
We will be particularly concerned with the order obtained by restricting $\cQ_{[n]}$ to the subsets of cardinalities $1$ and $k$.
Denote this poset by $\cQ_{[n]}^{1,k}$.
More generally, if $S \subseteq [n]$, we denote by $\cQ_{[n]}^S$ the restriction of $\cQ_{[n]}$ to subsets with cardinalities in $S$.
An useful observation is that $\dim \cQ_{[n]}^{1,k} = \dim \cQ_{[n]}^{[0,k]}$.
See for instance \cite{Lewis2021-em}*{Lemma 3.1} for a proof.

While $\dim \cQ_{[n]} \leq n$ follows trivially from the definition of dimension, equality holds.
In fact, Dushnik and Miller~\cite{Dushnik1941-ud} showed that $\dim \cQ_{[n]}^{1,n-1} = n$.
This was later extended by Dushnik~\cite{Dushnik1950-ma}, fully determining $\dim \cQ_{[n]}^{1,k}$ in the range $k \geq 2\sqrt{n}$.
Indeed, he showed that for $1 \leq r \leq \sqrt{n}$, we have $\dim \cQ_{[n]}^{1,k} \leq n - r$ if and only if $k \leq n/r + r - 3$.
In particular, if $k \geq 2\sqrt{n}$ then
\begin{equation}
\label{eq:dushnik}
    \dim \cQ_{[n]}^{1,k} \geq n - 2\sqrt{n}.
\end{equation}

Further results are asymptotic, and the known behaviour of $\dim \cQ_{[n]}^{1,k}$ can be divided the ranges between roughly $\log\log n$, $\exp(\sqrt{\log n})$ and $\sqrt{n}$.
When $k$ is quite small, that is, for $2 \leq k < \log_2 \log_2 n - \log_2 \log_2 \log_2 n$, we have
\begin{equation*}
    2^{k-2} \log_2 \log_2 n \leq \dim \cQ_{[n]}^{1,k} \leq k 2^k \log_2 \log_2 n.
\end{equation*}
The upper bound, which actually holds for all $k$, is due to Spencer~\cite{Spencer1971-fs}, while the lower bound is by Kierstead~\cite{Kierstead1996-kc}.

For the range $\log_2 \log_2 n < k < 2 \sqrt{n} - 2$, Kierstead~\cite{Kierstead1996-kc} showed that
\begin{equation}
\label{eq:kierstead-lower}
    \frac{k^2 \log n}{16 \log k} \leq \dim \cQ_{[n]}^{1,k}.
\end{equation}
In the same range, the following bounds of Füredi and Kahn~\cite{Furedi1986-bu}:
\begin{equation}
\label{eq:furedi-kahn}
    \dim \cQ_{[n]}^{1,k} \leq (k + 1)^2 \log n;
\end{equation}
and of Kierstead~\cite{Kierstead1996-kc}:
\begin{equation}
\label{eq:kierstead}
    \dim \cQ_{[n]}^{1,k} \leq 2 \p[\bigg]{\frac{k \log n}{\log k}}^2,
\end{equation}
are known.
While both \eqref{eq:furedi-kahn} and \eqref{eq:kierstead} holds for all $2 \leq k \leq n$, the first is better when $k < \exp(\sqrt{\log n})$.
Our application requires the following refinement of the Füredi-Kahn bound, which may be of independent interest.

\begin{theorem}
\label{thm:kahn-furedi-refined}
If $\cP$ is a suborder of $\cQ_{[n]}^{[1,k]}$ with $m$ elements then
\begin{equation}
\label{eq:kahn-furedi-refined}
    \dim \cP \leq (k+1) \ceil{\log n + \log m}.
\end{equation}
\end{theorem}
\begin{proof}
Consider a set $\cS$ of $d$ permutations of $[n]$, each chosen independently and uniformly at random.
For each permutation $\sigma \in \cS$, we associate a linear extension $L_\sigma$ of $\cQ_{[n]}^{[1,k]}$ as follows.
For two distinct elements $A,B \in \cQ_{[n]}^{[1,k]}$, let $x \in A \symdiff B$ be such that $\sigma(x)$ is maximal.
If $x \in B$, we set $A \leq_{L_\sigma} B$, otherwise set $B \leq_{L_\sigma} A$.
If $A \subseteq B$, then $A \symdiff B \subseteq B$, so $L_\sigma$ is indeed a linear extension.

Say that a pair $(x,A) \in [n] \times \cP$ with $x \notin A$ is \indef{covered} by $\cS$ if there is $\sigma \in \cS$ such that $A \leq_{L_\sigma} \set{x}$.
Suppose that $\cS$ covers every pair $(x,A) \in [n] \times \cP$ with $x \notin A$.
We claim that it follows that $\set{L_\sigma \st \sigma \in \cS}$ is a realiser for $\cP$.
Indeed, for incomparables $A, B \in \cP$, take $x \in B \setminus A$ and $\sigma \in \cS$ with $A \leq_{L_\sigma} \set{x}$.
Then by transitivity, we have $A \leq_{L_\sigma} B$, as $x \in B$.

Now we show that if $d = (k + 1)\ceil{\log n + \log m}$, with positive probability $\cS$ covers all pairs.
Indeed, fix a pair $(x,A) \in [n] \times \cP$ with $x \notin A$.
The probability that there is no $\sigma \in \cP$ with $A \leq_{L_\sigma} \set{x}$ is equal to
\begin{equation*}
    \p[\Big]{1 - \frac{1}{\card{A} + 1}}^d.
\end{equation*}
Consequently, by the union bound, the probability that some pair in $[n] \times \cP$ is not covered is at most
\begin{equation*}
    n m \p[\Big]{1 - \frac{1}{\card{A} + 1}}^d < n m \exp\p[\Big]{-\frac{d}{k+1}} = \exp\p[\Big]{\log n + \log m - \frac{d}{k+1}} \leq 1.
\end{equation*}
Thus, $\cS$ is realiser with positive probability.
This gives $\dim \cP \leq \card{\cS} = d$.
\end{proof}

From the proof, we can observe the same bound \eqref{eq:kahn-furedi-refined} holds if $m$ is the number of elements of $\cP$ that are not singletons.
As a corollary, we get Füredi-Kahn's result since $\cQ_{[n]}^{1,k}$ has $\binom{n}{k}$ elements that are not singletons and $\log \binom{n}{k} \leq k \log n$.

We will recover Kierstead's upper bound \eqref{eq:kierstead} in the next section with an improved constant, and discuss it's connection to $r$-cover-free families.


\section{Cover-free families}
\label{sec:cover}

Let $\cF$ be a family of subsets of $[n]$.
We say that $\cF$ is an \indef{$r$-cover-free} family if there are no distinct $F_0, F_1, \dotsc, F_r \in \cF$ with $F_0 \subseteq F_1 \union \dotsb \union F_r$.
Such families are relevant to us as they can be used to bound the dimension of some suborders of the hypercube.

\begin{proposition}
\label{prop:dimension-cover-cube}
If $\cF$ is an $k$-cover-free family consisting of $n$ subsets of $[d]$, then
\begin{equation*}
    \dim \cQ_{[n]}^{1,k} \leq d.
\end{equation*}
\end{proposition}
\begin{proof}
Write $\cF = \set{F_1, \dotsc, F_n}$.
Consider the map $\varphi$ from $\cQ_{[n]}^{1,k}$ to $\cQ_{[d]}$ defined by $\varphi(S) = \union_{i \in S} F_i$ for $S \subseteq [n]$, $\card{S} = 1, k$.
We claim that $\varphi$ is a poset embedding.
If $i \in S$ for some $i \in [d]$ and $S \subseteq [d]$, then clearly $\varphi(i) \subseteq \varphi(S)$, so $\varphi$ is monotone.
If $i \notin S$, then $F_i = \varphi(i) \not\subseteq \varphi(S) = \union_{j \in S} F_j$ as $\card{S} = k$ and $\cF$ is $k$-cover free.
Therefore, $\varphi$ is indeed an embedding.
As a consequence, we have $\dim \cQ_{[n]}^{1,k} \leq \dim \cQ_{[d]} = d$, as wanted.
\end{proof}

Denote by $f_r(n)$ be the maximum cardinality of an $r$-cover-free family of subsets of $[n]$.
This quantity has been studied in several ranges of $r$ and $n$.
A $1$-cover-free family is just an antichain, so Sperner's theorem~\cite{Sperner1928-hh} states that $f_1(n) = \binom{n}{\floor{n/2}}$.
For $r$ fixed and $n$ growing, Erdős, Frankl and Füredi~\cite{Erdos1985-js} showed that
\begin{equation}
\label{eq:bound-frn}
    \p[\Big]{1 + \frac{1}{4r^2}}^n \leq f_r(n) \leq e^{(1 + o(1))n/r}.
\end{equation}

For our applications, we will take $r$ to be roughly $n^{1/2}$.
In such range, the bounds in~(\ref{eq:bound-frn}) are very weak. Erdős, Frankl and Füredi also considered this range and proved the following result.

\begin{theorem}
\label{thm:cover-bound}
For $r = \eps n^{1/2}$ and $0 < \eps < 1/\sqrt{2}$, we have
\begin{align}
\label{eq:cover-bound}
    \p[\big]{1+o(1)}n^{(\floor{1/\eps} + 1)/2}
    \leq f_r(n)
    \leq n^{\ceil{2/\eps^2}}.
\end{align}
\end{theorem}

We will only make use of the lower bound, which Erdős, Frankl and Füredi prove through a beautiful construction, which we present here for the convenience of the reader.
Their lower bound also holds for the whole range of $r$, even though it is weaker than \eqref{eq:bound-frn} whenever $r = o(n^{1/2}/\log n)$.
The following form is more convenient for us.

\begin{proposition}
\label{prop:cover-bound}
As $n \to \infty$, we have
\begin{align*}
    f_r(n) \geq \exp\p[\bigg]{\p[\big]{1 - o(1)} \frac{n^{1/2} \log n}{2r} }.
\end{align*}
\end{proposition}
\begin{proof}
Let $h$ be an integer and $q$ be a power of a prime.
The ground set in which we construct our family is $X = \FF_q^2$, where $\FF_q$ is the field with $q$ elements.
Now consider
\begin{equation*}
  \cF = \set[\Big]{ \set[\big]{(x,g(x)) \st x \in \FF_q}
  \st g \in \FF_q[t],\, \deg g \leq h },
\end{equation*}
namely, the graphs of all polynomials with degree $\leq h$ in $\FF_q$.
This family has the property that each $F \in \cF$ has $\card{F} = q$ and for distinct $F, F' \in \cF$ we have $\card{F \inter F'} \leq h$, as if two polynomials of degree $\leq h$ have more than $h$ common points, they coincide.
Hence, if $F_0 \subseteq F_1 \union \dotsb \union F_r$ for distinct sets $F_i \in \cF$, we have $q \leq rh$.
So as long as $rh \leq q-1$, we cannot have such containment.
Therefore, $\cF$ is a $\floor{(q-1)/h}$-cover-free family with $q^{h+1}$ elements.

By the prime number theorem, there is a prime power $q \leq n^{1/2}$ with $ q = (1-o(1))n^{1/2}$.
Letting $h = \floor{(q-1)/r}$, we can construct an $r$-cover free family of size at least $q^{h+1} > \p[\big]{(1-o(1))n^{1/2}}^{(1-o(1))n^{1/2}/r} = n^{(1-o(1))n^{1/2}/2r}$ on a ground set of size $q^2 \leq n$.
\end{proof}

Using this construction, we also recover Kierstead's upper bound \eqref{eq:kierstead} with an improved constant.

\begin{corollary}
\label{cor:dimension-cover-cube}
As $n \to \infty$, we have
\begin{align*}
    \dim \cQ_{[n]}^{1,k} \leq \p[\big]{1 + o(1)}\p[\Big]{\frac{k \log n}{\log k}}^2.
\end{align*}
\end{corollary}
\begin{proof}
Let $\eps > 0$ and for $k < n \in \NN$, define
\begin{align*}
    d_{n,k} = \ceil[\Big]{\p[\Big]{(1 + \eps) 
 \frac{k \log n}{\log k}}^2}.
\end{align*}
If there is an $k$-cover-free family on $[d_{n,k}]$ of size at least $n$, then \Cref{prop:dimension-cover-cube} gives us that $\dim \cQ_{[n]}^{1,k} \leq d_{n,k}$.
Therefore, it is enough to show that $f_k(d_{n,k}) \geq n$.
By \Cref{prop:cover-bound}, we have
\begin{align*}
    \log f_{k}(d_{n,k}) 
    &\geq \p[\big]{1 - o(1)}\frac{d_{n,k}^{1/2} \log d_{n,k}}{2k}
    = \p[\big]{1 - o(1)}(1+\eps) \frac{ \log n \log d_{n,k}}{2\log k} \\
    &= \p[\big]{1 - o(1)}(1+\eps) \frac{ \log n}{\log k}\p[\Big]{\log k + \loglog n - \loglog k},
\end{align*}
which is more than $\log n$ if $n$ is sufficiently large.
Since $\eps > 0$ can be taken to be arbitrarily small, we are done.
\end{proof}


\section{Dimension of downsets of multisets}
\label{sec:downsets}

In this section we consider the combinatorial nature of suborders of the multiset order.
These considerations could be made specifically for the order in question, namely the divisibility order, but we believe that this would obfuscate the purely combinatorial nature of our argument.

Indeed, we provide generalisation of an argument in Lewis and Souza~\cite{Lewis2021-em}*{Lemma 3.1}, which is in part a generalisation of an argument by Dushnik~\cite{Dushnik1941-ud}, characterising the dimension of certain suborders of the hypercube in terms of suitable sets of permutations.

Given a set $X$, denote by $\cM_X$ the infinite poset of multisets with elements in $X$, ordered by containment with multiplicity.
For $A \in \cM_X$ and $x \in X$, denote by $\nu_x(A)$ the multiplicity of $x$ in $A$.
We denote by $\cQ_X$ the poset of \emph{subsets} of $X$ ordered by inclusion.
The \indef{support} of a multiset $A \in \cM_X$ is the set $s(A) \defined \set{x \in X \st \nu_x(A) > 0} \in \cQ_X$.
For a family $\cA \subseteq \cM_X$, the support of $\cA$ is the family $s(\cA) \defined \set{ s(A) \st A \in \cA} \subseteq \cQ_X$.
We then have $s(s(A)) = s(A)$ for a multiset $A \in \cM_X$, and $s(s(\cA)) = s(\cA)$ for a family of multisets $\cA \subseteq \cM_X$.

Denote by $X^{(k)}$ the family of subsets of $X$ with cardinality $k$.
We identify $X^{(1)}$ with $X$ by $\set{x} \mapsto x$.
A linear extension $L$ of $X^{(1)} \subseteq \cM_X$ can be identified with a permutation on $X$, that is, a bijection $\sigma \from X \to [n]$, with $n = \card{X}$.
Indeed, by setting $\sigma(x) = k$ when $x$ is the $k$-th element in the order $L$, we have $x \leq_L y$ if and only if $\sigma(x) \leq \sigma(y)$.

Fix $\cA \subseteq \cM_X$.
A set $\cS$ of permutations of $X$ is said to be \indef{$\cA$-suitable} if for every $A \in \cA$ and $x \in X$ with $x \notin A$, we have an element $\sigma \in \cS$ such that $\sigma(y) \leq \sigma(x)$ for every $y \in A$.
Denote by $M(\cA)$ the minimal cardinality of an $\cA$-suitable set of permutations.

Given a poset $P$, a \indef{downset} $F \subseteq P$ is a suborder that is \indef{down-closed}, namely, if $b \in F$ and $a \in P$ is such that $a \leq_P b$, then $a \in F$.
The following result allows us to compute the dimension of downsets in $\cM_X$ via sets of suitable permutations and, crucially, it shows that multiplicities do not interfere with the dimension of downsets.

\begin{proposition}
\label{prop:downset}
If $\cF \subseteq \cM_X$ is a downset, then
\begin{align*}
\dim \cF = M(s(\cF)).
\end{align*}
\end{proposition}
\begin{proof}
First of all, notice that we may assume that $X^{(1)} \subseteq \cF$, as otherwise, we could remove elements from $X$ until that is the case without changing either of the two quantities involved.
Let $\cL$ be a realiser for $\cF$ of size $\dim \cF$.
For each $L \in \cL$, let $\sigma_L$ be the restriction of $L$ to $X^{(1)}$, seen as a permutation on $X$.
Set $\cS \defined \set{\sigma_L \st L \in \cL}$.
We claim that $\cS$ is $s(\cF)$-suitable, which then implies that
\begin{equation*}
    M(s(\cF)) \leq \card{\cS}=\card{\cL}= \dim \cF.
\end{equation*}
By the definition of realiser, for each $A \in \cF$ and $x \notin A$, there is $L \in \cL$ with $A \leq_L \set{x}$.
But if $x \notin A$, then $x \notin s(A)$.
Hence, by transitivity, we have for each $y \in s(A)$ that $\set{y} \leq_{L} \set{x}$, and thus $\sigma_{L}(y) \leq \sigma_{L}(x)$.
This shows that $\cS$ is indeed $s(\cF)$-suitable.

Now let $\cS$ be a $s(\cF)$-suitable family of permutations of $X$ of size $M(s(\cF))$.
Given $\sigma \in \cS$, we define $L_\sigma$ as the colexicographic order on $\cF$ with respect to $\sigma$.
That is, for $A,A' \in \cF$, if $x \in X$ is the $\sigma$-greatest element with $\nu_x(A) \neq \nu_x(A')$, then we set $A <_{L_\sigma} A'$ if $\nu_x(A) < \nu_x(A')$, and $A' <_{L_\sigma} A$ otherwise.
We claim that $\cL \defined \set{L_\sigma \st \sigma \in \cS}$ is a realiser for $\cF$, which then implies that
\begin{equation*}
    \dim \cF \leq M(s(\cF)).
\end{equation*}
Indeed, for each $\sigma \in \cS$, $L_\sigma$ is a linear extension of $\cF$, as $A \subseteq A'$ implies $A \leq_{L_\sigma} A'$ and there are no incomparables in $L_\sigma$.
Moreover, if $A, A' \in \cF$ are incomparable, then there exists $x \in X$ with $\nu_x(A) > \nu_x(A')$.
Let $A'' = s(A') \setminus \set{x}$, and notice that $A'' \in s(\cF)$ as $\cF$ is a downset.
By definition of $s(\cF)$-suitable, there is some $\sigma \in \cS$ with $\sigma(y) \leq \sigma(x)$ for all $y \in A''$.
Therefore, the $\sigma$-greatest element $z \in X$ with $\nu_z(A) \neq \nu_z(A')$ is either $x$, in which case, $\nu_z(A) > \nu_z(A')$, or it is some other element $z \notin s(A') \subseteq s(A'') \union \set{x}$, in which case we also have $\nu_z(A) > \nu_z(A') = 0$.
In any case, we have $A' \leq_{L_\sigma} A$, where $A, A' \in \cF$ were arbitrary incomparables.
Thus $\cL$ is indeed a realiser for $\cF$.
\end{proof}

The next observation is that to compute the dimension of a downset $\cF \subseteq \cM_X$, it is enough to look at $s(\cF)$.

\begin{corollary}
\label{cor:support}
If $\cF \subseteq \cM_X$ is a downset, then
\begin{equation*}
    \dim \cF = \dim s(\cF).
\end{equation*}
\end{corollary}
\begin{proof}
If we consider the family $s(\cF)$ as a multiset in $\cM_X$, we get $s(s(\cF)) = s(\cF)$.
Therefore, we may apply \Cref{prop:downset} to see that
\begin{equation*}
    \dim s(\cF) = M(s(s(\cF))) = M(s(\cF)) = M(s(\cF)) = \dim \cF. \qedhere
\end{equation*}
\end{proof}

Our last remark about multisets allows us to partition the ground set $X$ when providing an upper bound on the dimension of a downset.

\begin{lemma}
\label{lem:decomposition}
Let $\cF \subseteq \cM_X$ be a downset and $X_1, \dotsc, X_k$ be a partition of $X$.
For each $Y \subseteq X$, write $\cF_Y \defined \set{A \in \cF \st s(A) \subseteq Y}$.
Then
\begin{equation*}
    \cF \embeds \cF_{X_1} \times \dotsb \times \cF_{X_k}.
\end{equation*}
\end{lemma}
\begin{proof}
Since $X_1, \dotsc, X_k$ is a partition of $X$, any set $A \in \cF$ can be written uniquely as $A = A_1 \union \dotsb \union A_k$, with $s(A_i) \subseteq X_i$.
Since $\cF$ is a downset, so is $\cF_{X_i}$ and thus $A_i \in \cF_{X_i}$.
Thus, the mapping $A  \mapsto (A_1, \dotsc, A_k)$ is well defined
and we claim that this is the poset embedding we need.
Indeed, if $A = A_1 \union \dotsb \union A_k$ and $B = B_1 \union \dotsb \union B_k$,
with $A_i,B_i \in \cM_{X_i}$,
then $A \subseteq B$ if and only if $A_i \subseteq B_i$ for all $i$.
\end{proof}

In \Cref{sec:divisibility} we specialise these results to the divisibility order.


\section{Dimension of the divisibility order}
\label{sec:divisibility}

The divisibility order on the whole positive integers, namely $\cD_\NN$, is identical to $\cM_\NN$ via the map that sends the positive integer $a$ to the multiset of its prime divisors, denoted here by $\cP(a)$, with multiplicity given by its prime factorisation.
That is, testing whether $a \divides b$ in $\cD_\NN$ is the same as testing whether $\cP(a) \subseteq \cP(b)$ in $\cM_\NN$.
Under the bijection $\cD_\NN \leftrightarrow \cM_\NN$, the support $s\p[\big]{\cP(a)}$ corresponds to the \indef{squarefree} part of $a$, namely, the product of the prime divisors of $a$ without multiplicity.
We can then simply write $s(a) = \prod_{p \in s(\cP(a))} p$ for the squarefree part of $a$.
For $A \subseteq \NN$, write $s(A) = \set{ s(a) \st a \in A}$.

We only use standard estimates for primes, like $\pi(x) = \p[\big]{1 + o(1)}x/\log x$ as $x \to \infty$, where $\pi(x)$ is the number of primes up to $x$, and $p_n = \p[\big]{1 + o(1)} n \log n$ as $n \to \infty$, where $p_n$ is the $n$-th prime.

Our main proof is similar in structure to that of \Cref{thm:lewissouza} in~\cite{Lewis2021-em}.
Indeed, the strategy consists of finding adequate posets $Q$ and $Q'$ such that $Q \embeds \cD_{[n]} \embeds Q'$, and hence, $\dim Q \leq \dim \cD_{[n]} \leq \dim Q'$.

We illustrate this by reproving the lower bound in \Cref{thm:lewissouza}.
Note that if $p_m^k \leq n$ then the map $S \mapsto \prod_{t \in S} p_t$ is a poset embedding $\cQ_{[m]}^{1,k} \embeds \cD_{[n]}$.
Let $\eps > 0$, and take $m = (1/16-\eps)\p[\big]{(\log n) / \loglog n}^2$ and $k = 2\sqrt{m}$.
It indeed follows that $p_m^k \leq n$ for large enough $n$.
Using Dushnik's lower bound \eqref{eq:dushnik}, we have 
\begin{equation*}
    \dim \cD_{[n]} \geq \dim \cQ_{[m]}^{1,2\sqrt{m}}
    \geq m - 2\sqrt{m} = \p[\big]{1/16 - \eps - o(1)} \frac{(\log n)^2}{(\loglog n)^2}.
\end{equation*}

The upper bound in \Cref{thm:lewissouza} is obtained by embedding $\cD_{[n]}$ in the product of multiple suborders of multisets, each one covering an interval of primes.
While this is also our strategy, our embedding will be more efficient than the previous one.

Given $A, X \subseteq \NN$, let $\cD_{A,X} = \p[\big]{\set{a \in A \st s(\cP(a)) \subseteq X}, \divides}$, that is, the divisibility order on the integers in $A$ with all prime factors in $X$.
Translating \Cref{lem:decomposition} to the divisibility order, we obtain the following.

\begin{corollary}[Lemma 4.1 in~\cite{Lewis2021-em}]
\label{cor:decomposition}
Let $\cP_1, \dotsc, \cP_k$ be a partition of $[n]$.
Then
\begin{equation*}
    \cD_{[n]} \embeds \cD_{[n],\cP_1} \times \dotsb \times \cD_{[n],\cP_k}.
\end{equation*}
\end{corollary}

Therefore, to upper bound $\dim \cD_{[n]}$, we can partition the primes in $[n]$ adequately and bound the contribution of each set of primes individually.

\begin{corollary}
\label{cor:squarefree}
The dimension of $\cD_{[n],X}$, the divisibility order on the integers in $[n]$ with prime factors in $X$, is the same as the dimension of $\cD_{s([n]),X}$, the divisibility order on the squarefree integers in $[n]$ with prime factors in $X$.
Indeed,
\begin{equation*}
    \dim \cD_{[n],X} = \dim \cD_{s([n]),X} = M(n,X).
\end{equation*}
\end{corollary}
\begin{proof}
$\cD_{[n],X}$ is a downset in $\cD_{\NN,X}$ identified with $\cM_X$.
The result follows from \Cref{cor:support} and \Cref{prop:downset}.
\end{proof}

The following bound for primes in the interval $(a,b]$ was implicit in~\cite{Lewis2021-em}:
\begin{equation}
\label{eq:implicit-bound}
  \dim \cD_{[n],(a,b]} = O\p[\bigg]{ \frac{(\log n)^2 \log b}{(\log a)^2} }.
\end{equation}
One key step in our proof is to get a better estimate than the one above.
The bound~\eqref{eq:implicit-bound} was obtained in~\cite{Lewis2021-em} by first embedding $\cD_{[n],(a,b]}$ into a poset of multisets ordered by inclusion with multiplicity, and then giving a bound on the dimension of that poset using the probabilistic method.
Equivalently, we can observe that if $x \in \cD_{[n],(a,b]}$, then $x$ cannot be the product of more than $(\log n) / (\log a)$ primes.
Therefore, the map $a \mapsto \cP(a)$ is a poset embedding from $\cD_{s([n]),(a,b]}$ into $\cQ_{[d]}^{[0,k]}$, where $k = (\log n)/(\log a)$ and $d = \pi(b) - \pi(a)$.
Füredi and Kahn's bound \eqref{eq:furedi-kahn} then gives \eqref{eq:implicit-bound}.

By using our refinement \Cref{thm:kahn-furedi-refined} instead, we obtain the following result.

\begin{proposition}
\label{prop:prime-interval}
The dimension of $\cD_{[n], (a,b]}$, the divisibility order on the integers in $[n]$ with prime factors in $(a,b]$, satisfies
\begin{align}
\label{eq:prime-interval}
    \dim \cD_{[n],(a,b]}
    \leq \p[\Big]{1 + \floor[\Big]{\frac{\log n}{\log a}}}\ceil[\big]{\log \pi(b) + \log n}.
\end{align}
In particular, if $a \leq n^{o(1)}$ as $n \to \infty$, then
\begin{align*}
    \dim \cD_{[n],(a,n]}
    \leq \p[\big]{2 + o(1)} \frac{(\log n)^2}{\log a}.
\end{align*}
\end{proposition}
\begin{proof}
Note we may assume $b \leq n$.
By \Cref{cor:squarefree}, we have $\dim \cD_{[n],(a,b]} = \dim \cD_{s([n]),(a,b]}$.
Denote by $X$ the set of primes in $(a,b]$.
Recall that the map $x \mapsto \cP(x)$ is a poset embedding from $\cD_{s([n]),(a,b]}$ into $\cQ_{X}$.
Moreover, if $x \in \cD_{s([n]),(a,b]}$, then $x$ is the product of at most $(\log n) / (\log a)$ primes.
Therefore, $x \mapsto \cP(x)$ embeds into $\cQ_{[d]}^{1,k}$ where $k = \floor{(\log n) / (\log a)}$ and $d = \pi(b) - \pi(a) \leq \pi(b)$.
Applying \Cref{thm:kahn-furedi-refined}, we obtain precisely \eqref{eq:prime-interval} as claimed.
\end{proof}

This bound will be one of the ingredients in the proof of \Cref{thm:main} in \Cref{sec:main}.
Indeed, we will use it with $a = \exp\p[\bib]{(\loglog n)^2}$, which gives
\begin{align*}
    \dim \cD_{[n],(a,n]}
    \leq \p[\big]{2 + o(1)} \p[\Big]{\frac{\log n}{\loglog n}}^2.
\end{align*}
With a slightly different choice of $a$, we can already improve the constant in \Cref{thm:lewissouza}.

\begin{corollary}
\label{cor:improved}
Let $\cD_{[n]}$ be the divisibility order on $[n]$.
Then, as $n \to \infty$,
\begin{equation*}
    \dim \cD_{[n]}
        \leq \p[\big]{1 + o(1)} \frac{(\log n)^2}{\loglog n}.
\end{equation*}
\end{corollary}
\begin{proof}
Let $A(n) = (\log n)^2 (\loglog n)^{-1}$.
By \Cref{cor:decomposition}, we have
\begin{align*}
    \dim \cD_{[n]}
        \leq \dim \cD_{[n],[1,A(n)]} + \cD_{[n], (A(n),n]}.
\end{align*}
But $\cD_{[n],[1,A(n)]}$ can be embedded into the product of $\pi(A(n))$ chains, so
\begin{align*}
    \cD_{[n],[1,A(n)]} \leq \pi(A(n))
    \leq \p[\big]{1 + o(1)} \frac{A(n)}{\log A(n)}
    = \p[\big]{1/2 + o(1)} \p[\Big]{\frac{\log n}{\loglog n}}^2.
\end{align*}
Finally, by \Cref{prop:prime-interval},
\begin{equation*}
\cD_{[n], (A(n),n]}
    \leq \p[\big]{2 + o(1)} \frac{ (\log n)^2}{\log A(n)}
    = \p[\big]{1 + o(1)} \frac{(\log n)^2}{\loglog n},
\end{equation*}
and we are done.
\end{proof}

To improve upon \Cref{cor:improved}, we will use the bound from \Cref{cor:dimension-cover-cube}, coming from $r$-cover-free families.

\begin{lemma}
\label{lem:dimension-boost}
Let $0 < \eps < 1/4$, let $d = d(n)$ be integer sequence with $\loglog d \leq \eps \loglog n$ and with $d \to \infty$ as $n \to \infty$. 
Then for large enough $n$, we have
\begin{equation*}
    \dim \cD_{[n],(d^{2^{k-1}}, d^{2^{k}}]} \leq \p[\big]{4 + 64\eps}\p[\Big]{\frac{\log n}{\loglog n}}^2,
\end{equation*}
for every integer $k$ with $1 \leq k \leq \eps \loglog n$.
\end{lemma}
\begin{proof}
Let $a = d^{2^{k-1}}$, $b = d^{2^{k}}$, $r = \floor{(\log n) / (\log a)}$, and $m = \pi(b) - \pi(a)$.
Recall from \Cref{cor:squarefree} that $\dim \cD_{[n],(a,b]} = \dim \cD_{s([n]),(a,b]}$.
Since $x \mapsto \cP(x)$ is a poset embedding of $\cD_{s([n]),(a,b]}$ into $\cQ_{[m]}^{1,r}$, we have $\dim \cD_{[n],(a,b]} \leq \dim \cQ_{[m]}^{1,r}$.
As $d \to \infty$ as $n \to \infty$, so does $m \to \infty$.
Therefore, \Cref{cor:dimension-cover-cube} gives us
\begin{align}
\label{eq:dimension-boost}
    \dim \cQ_{[m]}^{1,r}
    \leq (1+\eps)\p[\Big]{\frac{r\log m}{\log r}}^2
    \leq (1+\eps)\p[\Big]{\frac{\log n \log b}{\log a(\loglog n - \loglog a)}}^2,
\end{align}
for large enough $n$.

By the definition of $a$ and $b$, we have $\log b= 2 \log a$ and $\log a = 2^{k-1} \log d$. 
Moreover, since $k \leq \eps \loglog n$, we have
\begin{align*}
    \loglog a\leq \loglog d + k \log 2 \leq 2 \eps \loglog n ,
\end{align*}
for sufficiently large $n$.
Inserting this into \eqref{eq:dimension-boost}, we obtain for sufficiently large $n$ that
\begin{align*}
    \dim \cQ_{[m]}^{1,r}
    \leq \frac{4(1 + \eps)}{(1 - 2\eps)^2}\p[\Big]{\frac{\log n}{\loglog n}}^2
    \leq (4 + 64\eps)\p[\Big]{\frac{\log n}{\loglog n}}^2,
\end{align*}
where the last inequality holds as $0 < \eps < 1/4$, completing the proof.
\end{proof}

To improve the upper bound on $\dim \cD_{[n]}$, we will make use of \Cref{lem:dimension-boost} to bound the contribution of primes $\leq n$ in the range from $(\log n)^2$ to $(\log n)^{\loglog n}$ roughly.
In this interval, we have only need to apply \Cref{lem:dimension-boost} with $k = O(\logloglog n)$.


\section{Improving the upper bound}
\label{sec:main}

We can now finally proceed to prove the main result, namely, to provide a bound for $\dim \cD_{[n]}$ from above by $C (\log n)^2 (\loglog n)^{-2} \logloglog n$.

\begin{proof}[Proof of \Cref{thm:main}]
Fix an $0 < \eps < 1/4$.
Define $d(n) \defined \floor{(2 + \eps) (\log n) (\loglog n)^{-1}}$ and
\begin{equation*}
    K(n) \defined \ceil[\big]{(1/\log 2 + \eps) \logloglog n}.
\end{equation*}
Since $d \to \infty$ as $n \to \infty$, $\loglog d(n) = o(\loglog n)$ and $K(n) = o(\loglog n)$, we can choose $n_0 \defined n_0(\eps)$ sufficiently large so that \Cref{lem:dimension-boost} holds for all $1 \leq k \leq K(n)$, and such that for all $n \geq n_0$ we have that $2^{K(n)} \log d(n) > (\loglog n)^2$.
We now fix $n \geq n_0$.
In the following, we will write $d$ for $d(n)$ and $K$ for $K(n)$.

Let $A(n) \defined (\log n)^2 (\loglog n)^{-1}$ and $B(n) \defined \exp\p[\big]{(\loglog n)^2}$.
We decompose the primes in $[n]$ into three parts and apply \Cref{cor:decomposition} to obtain
\begin{align*}
    \dim \cD_{[n]}
        \leq \dim \cD_{[n],[1,A(n)]} + \dim \cD_{[n], (A(n),B(n)]} + \dim \cD_{[n], (B(n), n]}.
\end{align*}

But $\cD_{[n],[1,A(n)]}$ can be embedded into the product of $\pi(A(n))$ chains, so
\begin{align*}
    \dim \cD_{[n],[1,A(n)]} \leq \pi(A(n))
    \leq \p[\big]{1 + o(1)} \p[\Big] {\frac{\log n}{\loglog n}}^2.
\end{align*}

By \Cref{prop:prime-interval}, we have
\begin{equation*}
    \dim \cD_{[n],(B(n),n]}
    \leq \p[\big]{2 + o(1)}\frac{(\log n)^2}{\log B(n)}
    = \p[\big]{2 + o(1)} \p[\Big]{ \frac{\log n}{\loglog n}}^2.
\end{equation*}

To bound the contribution of the interval $(A(n), B(n)]$, we break it up further into more intervals, using \Cref{cor:decomposition} again.
We consider intervals of the form $(d^{2^{k-1}}, d^{2^k}]$ for $1 \leq k \leq K$.
Indeed, these intervals cover $(A(n), B(n)]$ as long as $d^{2^K} \geq B(n)$, which is the case as
\begin{align*}
    \log \p[\big]{d^{2^K}} = 2^K \log d > (\loglog n)^2 = \log B(n).
\end{align*}
Moreover, \Cref{lem:dimension-boost} gives
\begin{equation*}
    \dim \cD_{[n],(d^{2^{k-1}}, d^{2^{k}}]} \leq (4 + 64\eps)\p[\Big]{ \frac{\log n}{ \loglog n} }^2,
\end{equation*}
and therefore,
\begin{align*}
    \dim \cD_{[n], (A(n),B(n)]}
    &\leq K \p[\big]{4 + 64\eps} \p[\Big]{ \frac{\log n}{ \loglog n} }^2 \\
    &\leq \p[\big]{4 / \log 2 + 128\eps}  \frac{(\log n)^2 \logloglog n}{(\loglog n)^2}.
\end{align*}
Because $\eps$ can be chosen arbitrarily small, we have
\begin{equation*}
    \dim \cD_{[n]}
        \leq \p[\big]{4/\log 2 + o(1)}\,\frac{(\log n)^2 \logloglog n}{(\loglog n)^2},
\end{equation*}
as claimed.
\end{proof}


\section{Final remarks}
\label{sec:remarks}

Several variants of the notion of dimension for partial orders exists.
Most notably, the $2$-dimension of an order $P$ is the minimum $d$ such that $P$ can be embedded into the subsets of $[d]$ ordered by inclusion.
We denote it by $\dim_2 P$.
Lewis and the first author obtained bounds to $\dim_2 \cD_{[n]}$ similar to \Cref{thm:lewissouza} in~\cite{Lewis2021-em}, showing that
\begin{align*}
    \p[\big]{1/16 - o(1)} \frac{(\log n)^2}{(\loglog n)^2}
    \leq \dim_2 &\cD_{[n]}
    \leq \p[\big]{4e \pi^2 /3+ o(1)} \frac{(\log n)^2}{\loglog n}.
\end{align*}
We remark that our proof of \Cref{thm:main} does not extend to the $2$-dimension, chiefly because the reduction to square-free numbers does not work.
This of course does not imply that a similar bound couldn't hold for the $2$-dimension.
Indeed, it was conjectured in~\cite{Lewis2021-em} that $\dim \cD_{[n]}$ and $\dim_2 \cD_{[n]}$ are of the same order.
We now believe, however, that this conjecture could have been overly optimistic, and that it is very well reasonable that these two dimensions differ.

Regarding the original notion of dimension, we note that the difficulty lies in bounding the contribution of the primes in the range from $(\log n)^2$ to $(\log n)^{\loglog n}$.
In our proof, we used bounds on the dimension of $\cQ_{[n]}^{1,k}$ in the regime $\exp\p[\big]{(\log n)^{1/2}} \leq k \leq n^{1/2}$.
If the Kierstead lower bound \eqref{eq:kierstead-lower} turned out to be tight, that is, if we had
\begin{equation*}
    \dim \cQ_{[n]}^{1,k} = O\p[\bigg]{\frac{ k^2 \log n}{\log k}},
\end{equation*}
then then the proof of \Cref{lem:dimension-boost} can be adapted to show that 
\begin{equation}
\label{eq:unproven}
    \dim \cD_{[n],(d^{2^{k-1}}, d^{2^{k}}]} \leq C_k \p[\Big]{\frac{\log n}{\loglog n}}^2,
\end{equation}
with $C_k = \Theta(2^{-k})$.
Thus we could perform a more efficient covering of $\cD_{[n],(A(n), B(n)]}$ in the proof of \Cref{thm:main}, leading to $\dim \cD_{[n]} = \Theta\p[\big]{ (\log n / \loglog n)^2 }$.
In fact, a weaker bound of the form 
\begin{equation}
\label{eq:unproven-dim}
    \dim \cQ_{[n]}^{1,k} = O\p[\bigg]{k^2 \p[\Big]{\frac{\log n}{\log k}}^{2-\beta}},
\end{equation}
for some $\beta > 0$, would suffice to obtain a summable sequence $C_k$ in \eqref{eq:unproven}.
Even weaker savings would still suffice.

One way to obtain such an improvement is through the construction of more efficient $r$-cover-free families.
Indeed, if the upper bound of Erdős, Frankl and Füredi in \Cref{thm:cover-bound} is tight, that would imply that the aforementioned Kierstead lower bound \eqref{eq:kierstead-lower} on the $\dim \cQ_{[n]}^{1,k}$ is also tight.
As before, a weaker improvement on $f_r(n)$ would also suffice.
Indeed, $f_{r}(n) \geq n^{ (\floor{1 /\eps} + 1)/2}$ is improved to $f_{r}(n) \geq n^{1/ \eps^{1+\delta}}$ for $r = \eps n^{1/2}$ and some $\delta > 0$, that would also lead to \eqref{eq:unproven-dim} with $\beta > 0$, and consequently, to $\dim \cD_{[n]} = \Theta\p[\big]{ (\log n / \loglog n)^2 }$.

Conversely, if one can show that $\dim \cD_{[n]}$ grows faster than $(\log n / \loglog n)^2$, that would immediately obtain improved bounds lower bounds on $\dim \cQ_{[n]}^{1,k}$ in the regime $\exp\p[\big]{(\log n)^{1/2}} \leq k \leq n^{1/2}$ and improved upper bounds on $f_r(n)$ in the regime $r = \eps n^{1/2}$.

It is also of interest to consider the divisibility order $\cD_A$ for subsets $A \subseteq \NN$ other than $[n]$.
Lewis and Souza~\cite{Lewis2021-em}, showed that the bounds in \Cref{thm:lewissouza} also holds for $\dim \cD_{(n^\alpha, n]}$, with possibly different constants.
\Cref{thm:main} implies an improvement on the upper bound simply from the fact that $\cD_{(n^\alpha, n]}$ is a suborder of $\cD_{[n]}$.

The behaviour dramatically changes when we consider the divisibility order on intervals like $(n/\kappa, n]$, for some constant $\kappa > 1$.
Indeed, if $1 < \kappa \leq 2$, then $\cD_{(n/\kappa, n]}$ is an antichain, so the dimension is 2.
Lewis and Souza also showed that
\begin{equation*}
    \frac{c(\log \kappa)^2}{(\loglog \kappa)^2} \leq \dim \cD_{(n/\kappa, n]} \leq C \kappa (\log \kappa)^{1 + o(1)}.
\end{equation*}
The upper bound was significantly improved by Haiman~\cite{Haiman2022-mz}, who recently proved that
\begin{equation*}
    \dim \cD_{(n/\kappa, n]} \leq \frac{C (\log \kappa)^3}{(\loglog \kappa)^2}.
\end{equation*}


\section{Acknowledgements}

The authors are grateful for the helpful comments of Professor Béla Bollobás and for the invaluable suggestions of the referees.



\end{document}